\documentclass{amsart}
\usepackage{setspace}
\usepackage{a4}
\usepackage{amsthm,latexsym}
\usepackage{amsfonts}
\usepackage{graphicx}
\usepackage{textcomp}
\usepackage{cite}
\usepackage{enumerate}
\usepackage[mathscr]{euscript}
\usepackage{mathtools}
\newtheorem{theorem}{Theorem}[section]

\newtheorem{conjecture}[theorem]{Conjecture}

\newtheorem{definition}[theorem]{Definition}

\title{This is the title}
\usepackage{amssymb}
\usepackage{amsmath}
\usepackage{tikz}
\usepackage{hyperref}
\usepackage{mathtools}
\usetikzlibrary{cd}
\usepackage{fancyhdr}

\begin{document}
	\hrule\hrule\hrule\hrule\hrule
	\vspace{0.3cm}	
	\begin{center}
		{\bf{MODULAR DEUTSCH ENTROPIC UNCERTAINTY PRINCIPLE}}\\
		\vspace{0.3cm}
		\hrule\hrule\hrule\hrule\hrule
		\vspace{0.3cm}
		\textbf{K. MAHESH KRISHNA}\\
		School of Mathematics and Natural Sciences\\
	Chanakya University Global Campus\\
	Haraluru Village, Near Kempe Gowda International Airport (BIAL)\\
	Devanahalli Taluk, 	Bengaluru  Rural District\\
	Karnataka  562 110 India\\
	Email: kmaheshak@gmail.com\\
		
		Date: \today
	\end{center}

\hrule
\vspace{0.5cm}
\textbf{Abstract}: Khosravi, Drnov\v{s}ek  and Moslehian [\textit{Filomat, 2012}] derived Buzano inequality for Hilbert C*-modules. Using this inequality we derive Deutsch entropic uncertainty principle for Hilbert C*-modules over commutative unital C*-algebras.

\textbf{Keywords}:  Buzano inequality, Entropic uncertainty, Hilbert C*-module.

\textbf{Mathematics Subject Classification (2020)}: 46L08, 42C15, 46L05.
\vspace{0.5cm}
\hrule 

\section{Introduction}
Let $\mathcal{H}$ be a finite dimensional Hilbert space. Given an orthonormal basis  $\{\tau_j\}_{j=1}^n$ for $\mathcal{H}$, the \textbf{Shannon entropy}  at a point $h \in \mathcal{H}_\tau$ is defined as 
\begin{align}\label{D}
	S_\tau (h)\coloneqq - \sum_{j=1}^{n} \left|\left \langle h, \tau_j\right\rangle \right|^2\log \left|\left \langle h, \tau_j\right\rangle \right|^2,
\end{align}
where $\mathcal{H}_\tau\coloneqq \{h \in \mathcal{H}: \|h\|=1,  \langle h , \tau_j \rangle \neq 0, 1\leq j \leq n\}$ \cite{DEUTSCH}. In 1983, Deutsch derived following breakthrough entropic uncertainty principle for Shannon entropy  \cite{DEUTSCH}.
\begin{theorem}\cite{DEUTSCH} (\textbf{Deutsch Entropic Uncertainty Principle})  \label{DU}
	Let $\{\tau_j\}_{j=1}^n$,  $\{\omega_k\}_{k=1}^n$ be two orthonormal bases for a  finite dimensional Hilbert space $\mathcal{H}$. Then 
	\begin{align}\label{DEUTSCHUNC}
		 S_\tau (h)+S_\omega (h)\geq -2 \log \left(\frac{1+\displaystyle \max_{1\leq j, k \leq n}|\langle\tau_j , \omega_k\rangle|}{2}\right), \quad \forall h \in \mathcal{H}_\tau\cap \mathcal{H}_\omega.
	\end{align}
\end{theorem}
The inequality   (\ref{DEUTSCHUNC}) is recently derived for Banach spaces \cite{KRISHNA}.  It is observed very recently that using Buzano inequality (see \cite{BUZANO, FUJIIKUBO, STEELE})  one can provide a simple proof of Theorem \ref{DU} (see Corollary 1 in \cite{KRISHNA}). 
As Hilbert C*-modules became more important in noncommutative geometry, we are mainly motivated from the following problem. What is the modular version of Theorem \ref{DU}? Hilbert C*-modules are first introduced by Kaplansky \cite{KAPLANSKY} for modules over commutative C*-algebras and later developed for modules over arbitrary C*-algebras by Paschke  \cite{PASCHKE} and Rieffel \cite{RIEFFEL}.
\begin{definition}\cite{KAPLANSKY, PASCHKE, RIEFFEL}
	Let $\mathcal{A}$ be a  unital C*-algebra. A left module 	 $\mathcal{E}$  over $\mathcal{A}$ is said to be a (left) Hilbert C*-module if there exists a  map $ \langle \cdot, \cdot \rangle: \mathcal{E}\times \mathcal{E} \to \mathcal{A}$ such that the following hold. 
	\begin{enumerate}[\upshape(i)]
		\item $\langle x, x \rangle  \geq 0$, $\forall x \in \mathcal{E}$. If $x \in  \mathcal{E}$ satisfies $\langle x, x \rangle  =0 $, then $x=0$.
		\item $\langle x+y, z \rangle  =\langle x, z \rangle+\langle y, z \rangle$, $\forall x,y,z \in \mathcal{E}$.
		\item  $\langle ax, y \rangle  =a\langle x, y \rangle$, $\forall x,y \in \mathcal{E}$, $\forall a \in \mathcal{A}$.
		\item $\langle x, y \rangle=\langle y,x \rangle^*$, $\forall x,y \in \mathcal{E}$.
		\item $\mathcal{E}$ is complete w.r.t. the norm $\|x\|\coloneqq \sqrt{\|\langle x, x \rangle\|}$,  $\forall x \in \mathcal{E}$.
	\end{enumerate}
\end{definition}
Our prime tool to derive modular Deutsch uncertainty  is the following modular Buzano inequality by Khosravi, Drnov\v{s}ek,  and Moslehian \cite{KHOSRAVIDRNOVSEKMOSLEHIAN}.
\begin{theorem}\cite{KHOSRAVIDRNOVSEKMOSLEHIAN} (Modular Buzano Inequality)\label{MB}
If $\mathcal{E}$ is a Hilbert C*-module  over a  unital C*-algebra $\mathcal{A}$, then 
\begin{align*}
		\|\langle x, z \rangle \langle z,y \rangle\|\leq \frac{1}{2}\left(\|x\|\|y\|+\|\langle x, y \rangle\|\right), \quad \forall x,y,z \in \mathcal{E}, \langle z, z\rangle =1.
\end{align*}	
\end{theorem}
 In this paper we derive Theorem \ref{DU} for Hilbert C*-modules  over commutative unital C*-algebras.

\section{Modular Deutsch Entropic Uncertainty Principle}
	We begin by  recalling  the definition of frames for Hilbert C*-modules.
	\begin{definition}\cite{FRANKLARSON}
	Let 	 $\mathcal{E}$ be a Hilbert C*-module over a  C*-algebra $\mathcal{A}$. A collection $\{\tau_j\}_{j=1}^\infty $ in  $\mathcal{E}$ is said to be a (modular) \textbf{Parseval frame} for  $\mathcal{E}$ if 
	\begin{align*}
		 x  =\sum_{j=1}^\infty  \langle x, \tau_j \rangle \tau_j,   \quad \forall x \in \mathcal{E}.
	\end{align*}
\end{definition}
	A collection  $\{\tau_j\}_{j=1}^\infty $ in a Hilbert C*-module $\mathcal{E}$  over unital C*-algebra $\mathcal{A}$ with identity $1$  is said to have \textbf{unit inner product} if
	\begin{align*}
		\langle \tau_j, \tau_j \rangle =1, \quad \forall  j \in \mathbb{N}.
	\end{align*}
In analogy with Equation (\ref{D}), given a unit inner product Parseval frame $\{\tau_j\}_{j=1}^\infty $ for  $\mathcal{E}$, we define  \textbf{modular Shannon entropy}  at a point $x\in \mathcal{E}_\tau$ is defined as 
\begin{align}
	S_\tau (x)\coloneqq - \sum_{j=1}^{\infty}  \langle x, \tau_j\rangle \langle \tau_j, x \rangle \log ( \langle x, \tau_j\rangle \langle  \tau_j,x \rangle)
\end{align}
where $\mathcal{E}_\tau\coloneqq \{x \in \mathcal{E}: \langle x , x \rangle=1,  \langle x , \tau_j \rangle \neq 0, \forall j \in \mathbb{N}\}$.
\begin{theorem} (\textbf{Modular Deutsch Entropic Uncertainty Principle})
	Let 	 $\mathcal{E}$ be a Hilbert C*-module over a commutative unital C*-algebra $\mathcal{A}$. 	Let $\{\tau_j\}_{j=1}^\infty$,  $\{\omega_k\}_{k=1}^\infty$ be two Parseval frames  for  $\mathcal{E}$. Then 
		\begin{align*}
		S_\tau (x)+S_\omega (x)\geq -2 \log \left(\frac{1+\displaystyle \sup_{j,k \in \mathbb{N}}\|\langle\tau_j , \omega_k\rangle\|}{2}\right), \quad \forall x \in \mathcal{E}_\tau\cap \mathcal{E}_\omega.
	\end{align*}
\end{theorem}
\begin{proof}
	Let $x \in \mathcal{E}_\tau \cap \mathcal{E}_\omega$. Using the Parseval frame property,  the commutativity of C*-algebra,  Theorem \ref{MB} and the result  that  `function logarithm  is operator monotone' \cite{CHANSANGIAM}, we get 
	\begin{align*}
	S_\tau (x)+S_\omega (x)&=- \sum_{j=1}^{\infty}  \langle x, \tau_j\rangle \langle \tau_j, x \rangle \log ( \langle x, \tau_j\rangle \langle  \tau_j,x \rangle)- \sum_{k=1}^{\infty}  \langle x, \omega_k\rangle \langle \omega_k,  x \rangle \log ( \langle x, \omega_k\rangle \langle  \omega_k,x \rangle)\\
	&=-\sum_{j=1}^{\infty}\sum_{k=1}^{\infty} \langle x, \tau_j\rangle \langle \tau_j, x \rangle\langle x, \omega_k\rangle \langle \omega_k,  x \rangle \left[\log ( \langle x, \tau_j\rangle \langle  \tau_j,x \rangle)+\log ( \langle x, \omega_k\rangle \langle  \omega_k,x \rangle)\right]\\
	&=-\sum_{j=1}^{\infty}\sum_{k=1}^{\infty} \langle x, \tau_j\rangle \langle \tau_j, x \rangle\langle x, \omega_k\rangle \langle \omega_k,  x \rangle \log ( \langle x, \tau_j\rangle \langle  \tau_j,x \rangle \langle x, \omega_k\rangle \langle  \omega_k,x \rangle)\\
	&=-\sum_{j=1}^{\infty}\sum_{k=1}^{\infty} \langle x, \tau_j\rangle \langle \tau_j, x \rangle\langle x, \omega_k\rangle \langle \omega_k,  x \rangle \log ( \langle  \tau_j,x \rangle \langle x, \omega_k\rangle \langle  \omega_k,x \rangle \langle  x, \tau_j \rangle)\\
	&\geq-\sum_{j=1}^{\infty}\sum_{k=1}^{\infty} \langle x, \tau_j\rangle \langle \tau_j, x \rangle\langle x, \omega_k\rangle \langle \omega_k,  x \rangle \log \left( \frac{[\|\tau_j\|\|\omega_k\|+\|\langle \tau_j, \omega_k \rangle \|][\|\omega_k\|\|\tau_j\|+\|\langle  \omega_k, \tau_j \rangle \|]}{4}\right)\\
	&=-\sum_{j=1}^{\infty}\sum_{k=1}^{\infty} \langle x, \tau_j\rangle \langle \tau_j, x \rangle\langle x, \omega_k\rangle \langle \omega_k,  x \rangle \log \left( \frac{(\|\tau_j\|\|\omega_k\|+\|\langle \tau_j, \omega_k \rangle \|)^2}{4}\right)\\
	&=-2\sum_{j=1}^{\infty}\sum_{k=1}^{\infty} \langle x, \tau_j\rangle \langle \tau_j, x \rangle\langle x, \omega_k\rangle \langle \omega_k,  x \rangle \log \left( \frac{\|\tau_j\|\|\omega_k\|+\|\langle \tau_j, \omega_k \rangle \|}{2}\right)\\
	&\geq -2\sum_{j=1}^{\infty}\sum_{k=1}^{\infty} \langle x, \tau_j\rangle \langle \tau_j, x \rangle\langle x, \omega_k\rangle \langle \omega_k,  x \rangle \log \left( \frac{1+\|\langle \tau_j, \omega_k \rangle \|}{2}\right)\\
		&\geq -2\sum_{j=1}^{\infty}\sum_{k=1}^{\infty} \langle x, \tau_j\rangle \langle \tau_j, x \rangle\langle x, \omega_k\rangle \langle \omega_k,  x \rangle \log \left( \frac{1+\sup_{j,k \in \mathbb{N}}\|\langle \tau_j, \omega_k \rangle \|}{2}\right)\\
		&=-2\log \left( \frac{1+\sup_{j,k \in \mathbb{N}}\|\langle \tau_j, \omega_k \rangle \|}{2}\right)\sum_{j=1}^{\infty}\sum_{k=1}^{\infty} \langle x, \tau_j\rangle \langle \tau_j, x \rangle\langle x, \omega_k\rangle \langle \omega_k,  x \rangle\\
		&=-2\log \left( \frac{1+\sup_{j,k \in \mathbb{N}}\|\langle \tau_j, \omega_k \rangle \|}{2}\right) \langle x,  x \rangle\langle x,  x \rangle\\
		&=-2\log \left( \frac{1+\sup_{j,k \in \mathbb{N}}\|\langle \tau_j, \omega_k \rangle \|}{2}\right).
	\end{align*}
\end{proof}
In 1988, Maassen  and Uffink (motivated from the conjecture  by Kraus made in 1987 \cite{KRAUS}) improved Deutsch entropic uncertainty principle.
\begin{theorem}\cite{MAASSENUFFINK} 
	(\textbf{Maassen-Uffink Entropic Uncertainty Principle})  \label{MU}
	Let $\{\tau_j\}_{j=1}^n$,  $\{\omega_k\}_{k=1}^n$ be two orthonormal bases for a  finite dimensional Hilbert space $\mathcal{H}$. Then 
	\begin{align*}
		S_\tau (h)+S_\omega (h)\geq -2 \log \left(\displaystyle \max_{1\leq j, k \leq n}|\langle\tau_j , \omega_k\rangle|\right), \quad \forall h \in \mathcal{H}_\tau \cap \mathcal{H}_\omega.
	\end{align*}	
\end{theorem}
In 2013, Ricaud  and Torr\'{e}sani \cite{RICAUDTORRESANI} showed that orthonormal bases in Theorem \ref{MU} can be improved to Parseval frames. 
\begin{theorem}\cite{RICAUDTORRESANI} 
	(\textbf{Ricaud-Torr\'{e}sani Entropic Uncertainty Principle})  \label{RT}
	Let $\{\tau_j\}_{j=1}^n$,  $\{\omega_k\}_{k=1}^m$ be two Parseval frames  for a  finite dimensional Hilbert space $\mathcal{H}$. Then 
	\begin{align*}
		 S_\tau (h)+S_\omega (h)\geq -2 \log \left(\displaystyle \max_{1\leq j \leq n, 1\leq k \leq m}|\langle\tau_j , \omega_k\rangle|\right), \quad \forall h \in \mathcal{H}_\tau \cap \mathcal{H}_\omega.
	\end{align*}	
\end{theorem}
Proofs of  Theorems \ref{MU} and  \ref{RT} use Riesz-Thorin interpolation (RTI). To the best of author's knowledge, RTI  does not exists for abstract Hilbert C*-modules. Therefore we end by formulating the following conjecture. 
\begin{conjecture}
	(\textbf{Modular Kraus Entropic Conjecture}) 
		Let 	 $\mathcal{E}$ be a Hilbert C*-module over a commutative unital C*-algebra $\mathcal{A}$. 	Let $\{\tau_j\}_{j=1}^\infty$,  $\{\omega_k\}_{k=1}^\infty$ be two Parseval frames  for  $\mathcal{E}$. Then 
\begin{align*}
	 S_\tau (x)+S_\omega (x)\geq -2 \log \left(\displaystyle \sup_{j, k \in \mathbb{N}}\|\langle\tau_j , \omega_k\rangle\|\right), \quad \forall x \in \mathcal{E}_\tau \cap \mathcal{E}_\omega.
\end{align*}		
\end{conjecture}

\section{Acknowledgments}
This research was partially supported by the University of Warsaw Thematic Research Programme ``Quantum Symmetries".

  \bibliographystyle{plain}
 \bibliography{reference.bib}

\end{document}